\documentclass[10pt]{amsart}
\usepackage{amsmath,amssymb,amsthm,graphicx,tikz}
\usepackage{mathrsfs}
\usepackage{comment}
\usepackage{enumerate}
\usepackage[colorlinks=true]{hyperref}
\usepackage{xcolor}
\usepackage{cases}

\title[Response Solution]
{Response solutions to the
	quasi-periodically forced systems with degenerate equilibrium:
A simple proof of a result of W. Si and J. Si. and extensions}
\author[H.Cheng]{Hongyu Cheng}
\address{Chern Institute of Mathematics, Nankai University,
	Tianjin 300071, P.R.China.}
\email{hychengmath@126.com}
\author[R.de la Llave]{Rafael de la Llave}
\address{School of Mathematics, Georgia Inst. of Technology,
	Atlanta GA, 30332, USA}
\email{rafael.delallave@math.gatech.edu}
\thanks{R. L. is supported in part by NSF grant DMS 1800241. }
\author[F.Wang]{Fenfen Wang}
\address{
	School of Mathematics, Shandong University,
	Jinan, Shandong 250100, P.R.China.
}
\email{ffenwang@hotmail.com}
\thanks{F.W. is supported by CSC  by the National Natural Science
	Foundation of China (Grant Nos. 11171185,10871117), F.W.
	thanks G.T. for hospitality 2018-2019}
\date{\today}

\newtheorem{theorem}{Theorem}
\newtheorem{meta-thm}[theorem]{Meta-Theorem}
\newtheorem{lemma}[theorem]{Lemma}

\newtheorem{remark}[theorem]{Remark}
\newtheorem{definition}[theorem]{Definition}

\numberwithin{equation}{section}

\def\eps{\varepsilon}

\begin{document}

\maketitle

\begin{abstract}
We give a simple proof of  the
existence of response solutions in some quasi-periodically
forced systems with a degenerate fixed points.
The same questions were answered  in
\cite{ss18} using two versions of KAM theory.

Our method is based on reformulating  the existence of response
solutions as a fixed point problem in appropriate spaces of smooth
functions. By algebraic manipulations, the fixed point problem is transformed
into a contraction.

Compared to the KAM method,
the present method does not incur a loss of
regularity. That is,  the solutions we obtain have the same
regularity as the forcing. Moreover, the method here
applies when  problems are only finitely
differentiable. It also weakens  slightly the non-degeneracy conditions.
Since the method is based on the  contraction mapping principle, we also obtain automatically
smooth dependence on parameters and, when studying complex versions
of the problem we discover  the new phenomenon of monodromy.
We also present results for higher dimensional systems, but for higher
dimensional systems, the concept of degenerate fixed points is much
more subtle than in one dimensional systems.

To illustrate the power of the method, we also consider
two problems not studied in \cite{ss18}: the forcing
with zero average 
and second order oscillators. We show that in the zero average forcing case,
the solutions are qualitatively different, but the second order 
oscillators is remarkably similar.
\end{abstract}

\textbf{Keywords.} Degenerate fixed points; Second order oscillators;  Response  solutions; Fixed point theorem.

\textbf{2010 Mathematics Subject Classification.}  
35J70,  
42B20,  
34M35,  
34M45,  
34D15,  
37K55, 	
\section{Introduction}\label{sec:intro}

The goal of this paper is to
find  response solutions to quasi-periodically forced systems
with degenerate fixed points. The main technique we
use is the contraction mapping theorem in carefully chosen Banach spaces.

\subsection{The one-dimensional model}
The 1-D version of the problem
(the higher dimensional version of the problem will be formulated in
Section~\ref{sec:high}) is the following:
\begin{equation}\label{dege}
\begin{split}
 \dot{x}=x^l+h(\omega t,x) + \eps f(\omega t,x),\,\, x\in\mathbb{R},
\end{split}
\end{equation}
where $l\in \mathbb{N}$ with $l\geq 2$, $0<|\eps| \ll 1$ is a
small real parameter
(the small adaptations needed for considering
$\eps$ complex will be discussed in Section~\ref{sec:complex}), and
 $\omega$ is a vector in $\mathbb{R}^d$ with $d\in
\mathbb{N}$.  The function $h$ is assumed to vanish
in $x$ to order higher than $l$.
In the analytic case, vanishing to high order just  means
that $h(\theta, x) = x^{l+1} H(\theta, x)$ with $H$ an
analytic function. In the finitely differentiable case, we
will just need that  $ \partial_x^j h(\theta,0) = 0$
for $j = 0, 1,2,\ldots, l$ (we will also need that all the derivatives
up to a sufficiently high order are bounded for all $x$
in a neighborhood of the origin).

The functions will be assumed to have some regularity properties,
which we will detail once we have detailed the spaces in which
we will formulate the problem.

In our method the lower order terms do not play any important role
and can get incorporated in $f$ by scaling. We will keep it
in the model to facilitate the comparison with the paper \cite{ss18}
but we advice the reader that all the terms that come from it
will be subdominant.

The model \eqref{dege} represents physically the forcing
of a (one dimensional) fixed point which is degenerate.
We recall that \emph{``response solutions''} means solutions that have
the same frequency as the forcing.   The standard definition of
quasi-periodic functions are functions of time of the form
\eqref{realso}.
Hence, the problem we are considering is to produce solutions
of \eqref{dege}
of the form \eqref{realso}.

\subsection{Assumptions in the frequency}
Without loss of generality, we  assume that, for $\omega=(\omega_1,\cdots,\omega_{d})\in\mathbb{R}^d$,
\begin{equation}\label{non-res}
k \cdot \omega\neq 0,\,\,\,\text{for}
\,\, k=(k_1,\cdots,k_d)\in\mathbb{Z}^d\setminus\{0\},
\end{equation}
where $k \cdot \omega=\sum\limits_{i=1}^{d}k_i\omega_i$. Indeed, if there is a  $k_0\in \mathbb{Z}^d\setminus\{0\}$ such that $k_0\cdot \omega=0$, we could reformulate the forcing with only $(d-1)$-dimensional variables which are orthogonal to $k_0$.

In many related problems, one needs to assume not only
\eqref{non-res} but also lower bounds on $|k \cdot \omega|$.
It is remarkable that for the main results of this paper (and in \cite{ss18} )
the only requirement on $\omega$ is \eqref{non-res}.  Hence, the
results hold without any assumption in the frequencies.
In the study of some very degenerate
results (not considered in \cite{ss18}), we will impose some non-degeneracy conditions. Namely,
some rather weak Diophantine properties
\eqref{weakdiophantine} for the analytic case and the generally Diophantine properties for the 
finitely differentiable case. (See the Section~\ref{sec:noaverage}).

\subsection{The results in this paper}
We will produce two main results for model \eqref{dege}, one assuming
analytic regularity in the problem (see Theorem~\ref{mainthm})
and another one for finite regularity (see Theorem~\ref{theom-fi}). These two results are aimed at the real parameter $\varepsilon$.

 We will also consider the case of
complex parameter $\varepsilon$ and establish monodromy. (See Section~\ref{sec:complex} for more details).
Moreover, we will consider analogues of
\eqref{dege} in higher dimensions and establish results
in analytic and finite regularity. (See Theorem~\ref{theom-fih} in Section~\ref{sec:high}).

We will  also present results on the case of zero average forcing
and on oscillators, which are second order problems and, in
principle a singular perturbation. Remarkably, we obtain
that in the case of zero average, the solutions are qualitatively
different (see Section~\ref{sec:noaverage}), but in the oscillator case, the solutions are similar
to the solutions in the first order (see Section~\ref{sec:oscillators}).

\subsection{Relation to other papers}

The same problem was studied in many other papers.
In particular, it was studied in
\cite{ss18}, using  two versions of KAM theory.
We refer to the comprehensive introduction of \cite{ss18} for a
review of related literature  on the problem and other  methods used
to study it.

The method of this paper is very different from the method
of \cite{ss18} and the methods in other papers
referred in \cite{ss18}.  The basic idea of our method is that
we formulate
the existence of response solutions as functional equation,
which we manipulate till it becomes a fixed point in
an appropriate space of functions. Algebraic
manipulations transform the fixed point problem
into a fixed point for contractions.

We anticipate that, perhaps,
the most delicate step on our argument is the choice of
spaces since we want that they satisfy
several properties (see Section~\ref{sec:choice}).
Similar  methods  had also been used in
other response solution problems \cite{Rafael13,Rafael17,fenfen19}.
In particular, we will follow the notation of
\cite{fenfen19} and refer to that  paper for standard  technical details
(for example, well known properties of Sobolev spaces).

Eliminating the sophisticated KAM iteration allows us   to deal straightforwardly with cases in which the problem is
only finitely differentiable, and obtain automatically smooth dependence
on parameters. Also the solutions produced
have the same regularity as the forcing and we
do not incur the loss of
regularity that appears in KAM iteration.

The assumptions on the order of vanishing we use
is slightly weaker than in \cite{ss18}.
 We also weaken the non-degeneracy
assumptions in the case
that $l$ is even. We do not need to assume a sign for the
average, but in the even case, we need to restrict the values of
$\eps$.
 See the discussion of  \eqref{choosea}. In Section~\ref{sec:high} we
obtain analogues of the results in higher dimensions. Since
the proofs we present are based on soft methods, they also
work for infinitely dimensional problems.  The method allows
to discuss complex values of the parameters. The use of
the complex values for $\eps$
leads to  the new phenomenon of \emph{``monodromy"},
which we study in Section~\ref{sec:complex}.
We also consider some problems not considered in \cite{ss18}, 
namely, the case of zero average forcing (Section~\ref{sec:noaverage}) 
and second order degenerate  oscillators (Section~\ref{sec:oscillators}).

\subsection{Organization of this paper}\label{Organization}

This  paper is organized as follows: In Section~\ref{sec:formulation},
we present the main idea of reformulating the existence of response
solutions for equation \eqref{dege} as a fixed point problem. To solve
this fixed point equation, in Section~\ref{sec:choice}, we give the
precise function spaces that we work in and we list their important
properties, such as Banach algebra property and composition operator. We state our main
results and present the concrete proof in Section
~\ref{sec:analytic0}.  In Section~\ref{sec:complex} we study
the case of complex parameters and the monodromy phenomenon.  In Section~\ref{sec:high}, we deal with the
generally high-dimensional system. In Section~\ref{sec:noaverage}, we generalize the system \eqref{dege}
to the one whose forcing is zero average. In  Section~\ref{sec:oscillators}, we  study the degenerate second oscillators.
For the oscillators model, we just make some changes of variable to reduce this model to the one like \eqref{fixeq2} for model \eqref{dege}.

\section{Overview of the method in one-dimensional system}\label{sec:formulation}

In this section, we discuss heuristically the main ideas of our
treatment.  We will present in this section  only the formal manipulations
ignoring questions of domains etc. Those will be discussed later
but indeed, the formal manipulations of this section, will be
the motivations for the precise definitions later.

\subsection{A guide}
\label{sec:guide}
The manipulations
we perform are rather systematic and very common in nonlinear
analysis.   We firstly identify what we
expect to be the main part of the solution (in our case a
constant). If we write the unknown as  the guess plus an unknown
correction, we see that the original equation is equivalent to an equation
for the correction.
We furthermore observe that the  equation for the correction
has a main part that can be inverted, then, we
are left with a fixed point problem that has a good chance of
being a contraction. Of course, identifying what are the main
parts of the solution requires some experimentation (and some luck),
but checking that a guess is the correct one, can be done systematically.

\subsection{Some elementary notations}

For a function
$f: \mathbb{T}^d\times \mathbb{R}^n\rightarrow  \mathbb{R}^n$, we denote:

\begin{equation}\label{avera}
\begin{split}
&\overline{f}(x):=\int_{\mathbb{T}^d}f(\theta,x)d\theta, \\
& \widetilde{f}(\theta,x): = f(\theta,x) -\overline{f}(x).
\end{split}
\end{equation}
We refer to $\overline f$ as the average of $f$ with respect to $\theta$ and $\widetilde{f}$
as the oscillatory part of $f$.

We look for  quasi-periodic solutions with forcing frequency $\omega\in \mathbb{R}^d$. They are functions of time $t$ with the form
\begin{equation}\label{realso}
x(t) = a+V(\omega t),
\end{equation}
where $a\in \mathbb{R}$ is a number and $V:\mathbb{T}^d\rightarrow \mathbb{R}$ is a function to be determined.
Note that
representation of the function $x(t)$  is not unique.
{From} $a_1 +  V_1(\omega t) = a_2 + V_2(\omega t)$, we
can only conclude that $\widetilde{V_1} = \widetilde{V_2}$,
$a_1 - a_2 = \overline{V_2} - \overline{V_1}$.
\begin{remark}

A good heuristic guide  to guess that the dominative term in the response
function \eqref{realso} is a constant is the \emph{``averaging principle"}
(presented and partially justified in  \cite{Minorsky62,BogoliubovM61,Hale80}) which suggests that one
substitutes the forcing terms by their averages to obtain the leasing
approximations. Of course, the present paper can be considered as another
justification of the method.

In our case, the averaged equations of the system \eqref{dege} are:
\begin{equation*}
\dot{x}=x^l+\eps \overline{f}(x)
\end{equation*}
and the equilibrium is obtained by solving
$x^l+\eps \overline{f}(x)=0$, which we can further
approximate by $x^l+\eps \overline{f}(0)=0$.

Note that the case $\overline{f}(0)=0$ is a situation where the averaging principle does not provide any guidance and indeed, we will see that the leading
 part has a different form and, hence, the
solutions in this case are qualitatively different from
those with non-zero average forcing. (See Section~\ref{sec:noaverage}.)
\end{remark}

\begin{remark}
	Note that we depart slightly from the notation of \cite{ss18}.
We write the forcing as $\eps f(\omega t, x)$.  The paper \cite{ss18}
writes the forcing as $f(\omega t, x; \eps)$.

The paper \cite{ss18} presents two
main theorems about analytic functions.

Theorem $3.1$ in \cite{ss18}
assumes   Diophantine condition
\begin{equation}\label{ssdiophantine}
|k \cdot \omega| \ge \gamma/\Omega( |k|), \quad  \ln( \Omega(t))/t \rightarrow 0
\end{equation}
and a sign on the average.  We do not need any conditions in
$\omega$.

 In Theorem $3.2$ of \cite{ss18},
the Diophantine conditions \eqref{ssdiophantine} are eliminated, but there are two
new assumptions:
\begin{itemize}
\item
That the function agrees with the average to order
$\eps^2$, see $(3.6)$ in \cite{ss18}.  In our notation, this amounts to
$\widetilde{f}(\theta,0) = 0$ (we only need it is small enough).

\item
$h = \mathcal{O}(x^{2l})$, we assume
$h = \mathcal{O}(x^{l+1})$.
\end{itemize}
\end{remark}

\subsection{The invariance equations}
Substituting \eqref{realso} into equation \eqref{dege} and using that
 $\{\omega t\}_{t \in \mathbb{R}}$ is dense in $\mathbb{T}^d$,
  we obtain that
\eqref{dege} holds for a continuous function $x$ if and only if
$a$ and $V$ satisfy
 \begin{equation}\label{fixeq1}
 \begin{split}
  \left( \omega\cdot \partial_{\theta}\right) V(\theta)
= &(a+V(\theta))^l+ h(\theta, a + V(\theta)) + \eps f(\theta,a+V(\theta))\\
= & a^l+la^{l-1}V( \theta) +S(a,V(\theta))+
h(\theta, a + V(\theta) ) \\
&+ \eps \overline{f}(0)
+\eps\widetilde{f}(\theta,0)+\eps g(\theta,a+V(\theta)),
 \end{split}
 \end{equation}
 where
 \begin{equation}\label{remainder}
 \begin{split}
& S(a,V)=(a+V)^l-a^l-la^{l-1}V, \\
& g(\theta,x)=f(\theta,x)-f(\theta,0).
 \end{split}
 \end{equation}
Note that the equation \eqref{fixeq1} is slightly undetermined
because of the lack of uniqueness in the representation~\eqref{realso}.
This undetermination
will be useful for us.

\subsection{An important assumption}
A crucial assumption in our treatment (as well as that in  \cite{ss18}))
is $:$
\begin{equation}\label{averagenonzero}
\overline{f}(0)\neq 0.
\end{equation}

The importance of the assumption~\ref{averagenonzero} is
that the leading term in the response solution will be a constant.
Moreover, we will modify the method for the case that $\overline{f}(0)\neq 0$ to
study the situation when
$\overline{f}(0)= 0$ but the results (i.e. the form of the solutions) are
qualitatively different. (See Section~\ref{sec:noaverage}.)

\subsection{The leading term of the solution}

Our first step is to choose $a$ in \eqref{realso}  such that $:$
\begin{equation}\label{choosea}
a^l+ \eps \overline{f}(0)=0.
\end{equation}
Note that this choice is possible in several cases.
If $l$ is odd, we can find such an $a$ solving \eqref{choosea}
 for all  $\eps$ real.
If $l$ is even, we can find $a$ solving \eqref{choosea}
for all $\eps$ such that $\eps \overline{f}(0)$
has negative sign. Depending on the sign of  $\overline{f}(0)$
we obtain solutions in the positive real interval or in the negative
real interval.

In the even $l$ case, we obtain two solutions in the appropriate
interval of $\eps$. Each of them could be taken as the basis
to find the corretions $V$ so that we get two response solutions. As we vary
$\eps$, we obtain two branches of solutions.

We note that finding $a$ as above makes sense even for
values of $\eps$ which are complex, provided, of
course, that we allow for complex valued solutions.
 In Section~\ref{sec:complex},
we will take up the issue of complex values of $\eps$. The
use of complex values allows for much more topology and we
discover the phenomenon of \emph{``monodromy''}.

Once we have accomplished finding an $a$ which eliminates
several terms in \eqref{fixeq1}, we study the remaining
equation. We find it convenient to
introduce the linear operator :
\begin{equation}\label{linearop}
\begin{split}
\mathcal{L}_{a}:=\omega\cdot \partial_{\theta}-la^{l-1}
\end{split}
\end{equation}
defined on one-dimensional periodic functions of $\theta\in \mathbb{T}^d$.

\subsection{The equation for the corrections}
Using the choice of $a$ in \eqref{choosea}
and the  notation  \eqref{linearop},
we see that the equation  \eqref{fixeq1} is equivalent to
the following equation for $V$ :
\begin{equation}\label{fixeq2}
\begin{split}
&\mathcal{L}_{a}(V(\theta)) \\
&\ \ \ \ =S(a,V(\theta)) + h(\theta, a + V(\theta)) + \eps \widetilde{f}(\theta,0)+\eps g(\theta,a+V(\theta)).
\end{split}
\end{equation}

If  we select spaces in which $\mathcal{L}_{a}$ is
boundedly invertible, then the
equation \eqref{fixeq2} can be transformed into :
\begin{equation}\label{fixeq3}
\begin{split}
V(\theta)
&=\mathcal{L}_{a}^{-1}\left( S(a,V(\theta))
+h(\theta, a +V(\theta) )
+\eps\widetilde{f}(\theta,0)+\eps g(\theta,a+V(\theta))\right)\\
& \equiv \mathcal{T}_{a}(V)(\theta).
\end{split}
\end{equation}

We will show that we can apply the contraction mapping
principle to the equation \eqref{fixeq3} once we identify appropriate
Banach spaces and a ball in them mapped to itself by the operator
$\mathcal{T}_a$ defined in \eqref{fixeq3}.
In the following section, we will make the choice of spaces explicitly.

\section{Choice of spaces and some preliminary results
on them}\label{sec:choice}

To make precise the calculations in Section~\ref{sec:formulation},
we just need to choose appropriate function spaces and
check that we can carry the steps indicated formally there and
indeed obtain  that $\mathcal{T}_a$ is a contraction in a ball.

\subsection{Some preliminary considerations}

There are a few  guiding principles in the choice of spaces :
\begin{itemize}
\item
The norms of the functions in the spaces
can be read off from the size of the Fourier coefficients.
In such a way, the norm of the operator $\mathcal{L}_a$ defined in \eqref{linearop},
which is diagonal in Fourier series, can be estimated very precisely
from one space in the class to itself.
\item
The spaces have to possess good Banach algebra properties for
multiplication so that one can perform nonlinear analysis.
\item
The operator of composition in the left can be
estimated.
\end{itemize}

With the above considerations, it is reasonable to
consider the following well known spaces which have
been found useful in many nonlinear problems (in particular,
they were used in problems similar to ours in \cite{Rafael13,Rafael17,fenfen19}).

\subsection{Some standard spaces we will use}
\label{sec:concretespaces}

For $\rho \geq 0$, we denote by
\begin{equation*}
\mathbb{T}^{d}_{\rho}=\left\lbrace \theta\in \mathbb{C}^d/(2\pi\mathbb{Z})^d\,:\, \mathrm{Re}(\theta_j)\in \mathbb{T},\,\,|\mathrm{Im}(\theta_j)|\leqslant \rho,\,\,j=1,\ldots,d \right\rbrace.
\end{equation*}
We denote the Fourier expansion of a periodic function
 $f(\theta)$ on  $\mathbb{T}^{d}_{\rho}$ by
\begin{equation*}
\begin{split}
f(\theta)=\sum_{k\in\mathbb{Z}^{d}}\widehat{f}_{k}e^{\mathrm{i} k\cdot\theta},
\end{split}
\end{equation*}
where $k\cdot\theta=\sum_{i=1}^d k_i \theta_i$ represents the Euclidean product in $\mathbb{C}^d$ and $\widehat{f}_{k}$ are the Fourier coefficients of $f$.

\begin{definition}\label{space}
For $\rho \geq 0,\,m \in
\mathbb{N}$, we denote by	$H^{\rho,m}$ the space of analytic functions $V$ in $\mathbb{T}_\rho^d$ with finite norm $:$
	\begin{equation*}
	\begin{aligned}
	H^{\rho,m}:&=H^{\rho,m}(\mathbb{T}^d,\mathbb{C}^n)\\
	&=\left\lbrace V:\,\mathbb{T}_\rho^d\rightarrow \mathbb{C}^n\,\mid\,\|V\|_{H^{\rho,m}}^{2}=\sum_{k\in \mathbb{Z}^d}|\widehat{V}_{k}|^{2}e^{2\rho |k|}
	(|k|^{2}+1)^{m}<+\infty\right\rbrace.
	\end{aligned}
	\end{equation*}
\end{definition}

It is obvious that the space $\left(H^{\rho,m},\,\,\|\cdot
\|_{H^{\rho,m}}\right)$ is a Banach space and indeed a Hilbert space. From
the real analytic point of view, we consider the Banach space
$H^{\rho,m}$ of the functions that take real values for real
arguments. This is  Banach space over the reals.

For $\rho=0$, $H^{m}:=H^{0,m}(\mathbb{T}^d,\mathbb{R}^n)$ is the
standard Sobolev space, we refer to the references
\cite{taylor3} for more details.
Moreover, when
$m>\frac{d}{2}$, by the Sobolev embedding theorem, we obtain that
$H^{m+p}(\mathbb{T}^d,\mathbb{R}^n)\,(p=1,2,\cdots)$ embeds continuously into $C^{p}(\mathbb{T}^d,\mathbb{R}^n)$.

For $\rho>0$, functions in the space $H^{\rho,m}$ are analytic in the
interior of $\mathbb{T}_\rho^d$ and extend to Sobolev  functions
of order $m$ on the
boundary of $\mathbb{T}_\rho^d$.   For $m > d$, the space $H^{\rho, m}$ can also
be considered as closed subspace of the Sobolev space in the $2d$-dimensional real  manifold
with boundary $\mathbb{T}^d_\rho$.

\subsection{Some standard properties of the Sobolev spaces $H^{\rho,m}$}
\label{sec:properties}

It is well known that the Sobolev spaces $H^{\rho,m}$ defined above
satisfy the Banach algebra property for large enough $m$ (we refer to
\cite{taylor3} for more details).

\begin{lemma}[Banach algebra properties]\label{alge}
	We have the following properties in two cases:
	\begin{enumerate}
		\item -\emph{Sobolev case :} For $\rho=0,\,m>\frac{d}{2}$, there exists a constant $C_{m,d}>0$ depending only on $m,d$ such that for
		$V_1,\,V_2\in H^{m}$, the product $V_1\cdot V_2\in H^{m}$ and
		\begin{equation*}
		\|V_1V_2\|_{H^{m}}\leq C_{m,d}\|V_1\|_{H^{m}}\|V_2\|_{H^{m}}.
		\end{equation*}
		\item -\emph{Analytic case :} For $\rho>0,\,m>d$, there exists a constant $C_{\rho,m,d}>0$ depending  on $\rho,m,d$ such that for
		$V_1,\,V_2\in H^{\rho,m}$, the product $V_1\cdot V_2\in H^{\rho,m}$ and
		\begin{equation*}
		\|V_1V_2\|_{H^{\rho,m}}\leq C_{\rho,m,d}\|V_1\|_{H^{\rho,m}}\|V_2\|_{H^{\rho,m}}.
		\end{equation*}
	\end{enumerate}
	
	In particular,  $H^{\rho,m}$ is a Banach algebra when $\rho,\,m,\,d$ are as above.
\end{lemma}

It is interesting to remark that the value of $m$ is what controls the Banach
algebra properties (which are crucial for us). On the other hand, for
regularity, the parameter $\rho$ is much more relevant. For a KAM argument,
one could use many different sets of spaces since the Newton method would
overcome all these difficulties.  The present method
of using only a contraction argument is much
more restrictive on the spaces we use since we cannot loose any regularity
in the iterative step and we also need some Banach algebra properties.

 The
Banach spaces $H^{\rho,m}$ seem a good compromise between having
norms given by Fourier coefficients (which makes the linear estimates
efficient) and  having Banach algebra properties. They are also Hilbert spaces
which makes spectral theory particularly powerful. These properties
have been found useful in several areas such as quantum field theory.

The following results on composition are also rather standard.

\begin{lemma}[Composition properties]\label{gag-nir00}

Assume that $\rho > 0$.
	
		Let $g:\mathbb{T}^d_{\rho}\times B \rightarrow \mathbb{C}^n $
with $B$ being an open ball around the origin in $\mathbb{C}^n$ and assume that $g$
		is analytic  in $\mathbb{T}_{\rho}^d\times B$.

Then, for $V\in H^{\rho,m}(\mathbb{T}_{\rho}^d,\,\mathbb{C}^n)\cap L^{\infty}(\mathbb{T}_{\rho}^d,\,\mathbb{C}^n)$ with $V(\mathbb{T}_{\rho}^d)\subset B$, we have
		\begin{equation}\label{comsob0}
		\begin{split}
		\|g(\theta,V)\|_{ H^{\rho,m}}\leq C_{\rho ,m}\left(\|V\|_{L^{\infty}}\right)\left(1+\|V\|_{ H^{\rho,m}}\right).
		\end{split}
		\end{equation}
		Moreover, when $m>d$,
		\begin{equation}\label{diffremainder}
		\begin{split}
		\|g(\theta,V+W)-g(\theta, V)&-D_Vg(\theta,V) \cdot W\|_{ H^{\rho,m}}\\
		&\leq C_{\rho,m,d}\left(\|V\|_{L^{\infty}}\right)\left(1+\|V\|_{ H^{\rho,m}}\right)\|W\|_{ H^{\rho,m}}^2.
		\end{split}
		\end{equation}

In the case that $\rho = 0$, it suffices to assume that $g \in C^{m+2}$
in  real neighborhood and that $m > d/2$. Then, we have
\eqref{comsob0}, \eqref{diffremainder}.
\end{lemma}

The results in Lemma~\ref{gag-nir00} are somewhat
standard. For the sake of completeness, we give some
sketch.  Many details, counterxamples for related
statements, etc. are in \cite{AZ90,kappe03,taylor3} or  in \cite{fenfen19}.

The reason for  the inequality
\eqref{diffremainder} is that, by the fundamental theorem of calculus
\begin{equation*}
\begin{split}
g(\theta,V(\theta)+W(\theta))&-g(\theta, V(\theta))-D_Vg(\theta,V(\theta)) \cdot W(\theta)\\
&=\int_{0}^1\int_{0}^{1} tD^2_Vg(\theta,V(\theta)+stW(\theta))\cdot W^2(\theta)dsdt.
\end{split}
\end{equation*}

Then,
we get the desired result by the facts that $D^2_Vg(\theta,V(\theta)+stW(\theta))\in H^{\rho,m}$ and its $H^{\rho,m}$ norm is bounded uniformly in $t,s$ and that $H^{\rho,m}$ is a  Banach algebra under multiplication by  Lemma~\ref{alge} and using \eqref{comsob0} for the second derivative.

To establish the standard
inequality  \eqref{comsob0}, it suffices to use the Faa di Bruno formula
for derivatives and then, the Moser-Nierenberg inequalities for
products of derivatives.

\begin{remark} \label{rangebounds}
We call attention that we are considering only the cases when
the Sobolev embedding theorem applies and the functions we are considering
are bounded. This allows the consequence that the bounds
in \eqref{diffremainder} are the bounds of the derivatives of
$g$ in the range of the functions considered.
\end{remark}

\begin{remark}\label{highder}
Note that, in the case of analytic regularity, \eqref{diffremainder}  establishes that the
left composition operator  $\mathcal{C}_g(V)(\theta): =  g(\theta,
V(\theta))$, considered as a function from the space $H^{\rho,m}$ to itself, is
differentiable. This shows
that the composition operator $\mathcal{C}_g$ is analytic.

 Note also that \eqref{diffremainder}
establishes that the derivative is the multiplication
by the another left composition. Hence, we can apply the same
result to obtain higher differentiability properties (under
appropriate hypothesis).
This shows that if $g \in C^{p+m +2}$ with $p=0,1,\cdots$ and $m > d/2$,
the left composition operator $\mathcal{C}_g$
is $C^{p+1}$ acting on the space $H^m$. We refer to \cite{AZ90}.
\end{remark}

\section{Existence of response solutions for one-dimensional system}
\label{sec:analytic0}

In this section, we implement the strategy discussed at the beginning of
Section~\ref{sec:formulation} using the spaces discussed in Section~\ref{sec:choice}.

\subsection{Analytic case}\label{sec:analytic}

In this section, we state the main result and the corresponding proof for the
model \eqref{dege} in which the forcing is analytic.
\begin{theorem}\label{mainthm}
We study the equation \eqref{dege} with $h$ vanishing to order $(l+1)$ at zero.

Assume that $f,\,h$ are
analytic in $\mathbb{T}_{\rho}^d\times B$ with $B$ being an open
ball around the origin in the space $\mathbb{C}$ and
$\widetilde{f}(\theta,0)\in H^{\rho,m}(\mathbb{T}^d,\mathbb{C})$ for
some $\rho>0,\,m>d$.

If \eqref{averagenonzero} holds
and  $\|\widetilde f(\theta,0)\|_{H^{\rho,m}}$ is small enough compared
to $|\overline{f}(0)|$, then, there exists a $\eps_0 > 0 $ such that, defining
$\mathcal{I} = (-\eps_0, \eps_0)$ for $l$ odd
and $\mathcal{I} = (-\eps_0, 0)$ when $l$ even and
$\overline{f}(0) > 0$,
and $\mathcal{I} = (0, \eps_0)$ when $l$ even and
$\overline{f}(0) < 0$,
we have
that for all
$\eps \in \mathcal{I}$, there
	exists a solution of \eqref{dege} of the form \eqref{realso} in $H^{\rho,m}$.

Moreover, the solution for equation \eqref{dege} is locally unique.
\end{theorem}

By reading the proof in the following part, we obtain explicit estimates of
the domain where local uniqueness holds. Roughly, they are domains of
size $\approx | \eps|^{1/l}$. This is consistent with the fact that
in the case that $l$ is even we obtain several solutions at this
distance (or when we consider complex valued solutions).

When we have a locally unique solution for all values of $\eps$, we
can discuss the regularity with respect to the parameter $\eps$. It
follows that one can get that it is analytic in $\eps$.

\begin{remark}
Note that the above regularity statement, does not include regularity
at $\eps = 0$. This seems possible under some very weak Diophantine
properties such as \eqref{weakdiophantine}. One obtains
an approximate solution as a polynomial in
 $\eps$ and starts a contraction mapping around it.
We omit a precise formulation and a proof. See \cite{fenfen19,Rafael17}.
\end{remark}

\subsubsection{Proof of Theorem~\ref{mainthm}}

In this section, we prove  Theorem~\ref{mainthm} by considering the fixed point  equation \eqref{fixeq3} in Banach space $H^{\rho,m}$.

It is easy to obtain
 the quantitative bounds on the  inverse of $\mathcal{L}_a$ defined in  \eqref{linearop} with $a$ being the one in \eqref{choosea}, as an operator from the space $H^{\rho,m}$ to itself.
Indeed, when we write a function $V(\theta)\in H^{\rho,m}$ in the Fourier expansion as
 \begin{equation*}
V(\theta) =\sum_{k\in\mathbb{Z}^{d}}\widehat{V}_{k}e^{\mathrm{i} k\cdot\theta},
 \end{equation*}
 the operator $\mathcal{L}_{a}$ acting on the Fourier basis becomes
\begin{equation*}
\begin{split}
\mathcal{L}_{a}(e^{\mathrm{i} k\cdot\theta})=\left( \mathrm{i}(k\cdot \omega) -la^{l-1}\right) e^{\mathrm{i} k\cdot\theta}=:L_a(k\cdot \omega)e^{\mathrm{i} k\cdot\theta}
\end{split}
\end{equation*}
with $L_a(k\cdot \omega)=\mathrm{i}(k\cdot \omega) -la^{l-1}$.

Due to the fact that the norm in the space $H^{\rho,m}$ is  characterized by the Fourier coefficients, we obtain that
\begin{equation}\label{inversbound}
\begin{split}
\left\|\mathcal{L}_{a}^{-1}\right\|_{H^{\rho,m}\rightarrow H^{\rho,m}}& =\sup_{k\in \mathbb{Z}^d}\left|L_a^{-1}(k\cdot \omega)\right|
=\sup_{k\in \mathbb{Z}^d}\frac{1}{\left|\mathrm{i}(k\cdot \omega) -la^{l-1}\right|}\\
&\leq \frac{1}{\left|la^{l-1}\right|}
=\frac{1}{l|\overline{f}(0)|^{1-\frac{1}{l}} } |\eps|^{-1+\frac{1}{l}}.
\end{split}
\end{equation}
\begin{remark}
In Section~\ref{sec:complex}, we will consider the case that $a$ is
complex.

We remark that when $a$ is complex, we have, by the same
argument
\begin{equation}\label{complexinverse}
\left\|\mathcal{L}_{a}^{-1}\right\|_{H^{\rho,m}\rightarrow H^{\rho,m}}
\le 1/\mathrm{Re}(la^{l-1})= 1/\emph{dist}(la^{l-1}, \mathrm{i} \mathbb{R}),
\end{equation}
where $\emph{dist}(la^{l-1}, \mathrm{i} \mathbb{R})$ is the distance
between $la^{l-1}$ and $\mathrm{i} \mathbb{R},$ which is $\mathrm{Re}(la^{l-1}).$

For simplicity, we will omit the subscript of $\|\mathcal{L}_{a}^{-1}\|_{H^{\rho,m}\rightarrow H^{\rho,m}}$ and use the notation $\|\mathcal{L}_{a}^{-1}\|$ in the following. We also simplify the notation $\|\cdot\|_{H^{\rho,m}}$   as $\|\cdot\|_{\rho,m}$ when there is no confusing.
\end{remark}

We now look for a fixed point for the operator $\mathcal{T}_a$ defined in
\eqref{fixeq3}.
Consider a ball $\mathcal{B}_{r}(0)$ around the origin in $H^{\rho,m}$ with radius $r>0$. We will show that one can obtain
$r$ such that $\mathcal{T}_{a}(\mathcal{B}_{r}(0))\subset \mathcal{B}_{r}(0)$  and $\mathcal{T}_{a}$ is a contraction on $\mathcal{B}_{r}(0)$.
\medskip

For $S(a,V),\,g(\theta,a+V)$ defined in \eqref{remainder}, by the fact that  one has
that the Lipschitz constant of the the nonlinear terms over
a ball with radius $r$ small is
\begin{equation*}
\begin{split}
&\text{Lip}_{V}(S)\leq C|a|^{l-2}r,\\
&\text{Lip}_{V}(h)\leq C(|a|+r)^l,\\
& \text{Lip}_V(g)\leq C,
\end{split}
\end{equation*}
where $\text{Lip}_V(h)$ denotes the Lipschitz constant of
$h(\theta,V)$ with respect to the second argument $V$
in the ball of radius $r$,  and
$C$ is a positive constant depending on
$l$ and $f$.

Note that in the contraction arguments, there are
two conditions (that the ball gets mapped into itself
and that the map is a contraction in the ball). We obtain two
results: existence and uniqueness. The uniqueness result is
stronger taking large balls and the existence is stronger for
smaller balls. Hence, it is good to have some flexibility.

\medskip
For any $V_1,\,V_2\in\mathcal{B}_{r}(0)$, we have, assuming that
$r$ is small (remember that $a$ is given by \eqref{choosea})
\begin{equation} \label{lipschitztau}
\begin{split}
\|\mathcal{T}_a(V_1)&-\mathcal{T}_a(V_2)\|_{\rho, m}
\\
&\leq \|
\mathcal{L}_{a}^{-1}\|\big(\text{Lip}_V(S)
+ \text{Lip}_V(h) +
|\eps|\text{Lip}_V(g) \big)\|V_1-V_2\|_{\rho, m}\\
&\le
C  \| \mathcal{L}_{a}^{-1} \| \big( |a|^{l-2} r  +  (|a| +r)^{l}+
|\eps|\big) \| V_1 - V_2 \|_{\rho, m}.
\end{split}
\end{equation}
Note that we have used Remark~\ref{rangebounds} to take advantage of
the fact that some functions appearing in $\mathcal{T}_a$ vanish
to a high order.   The most delicate term above is the derivative
of $S$ which takes advantage of $S$ not only being second order in $V$
but also $a$ being small.

Taking $|a| \approx |\eps|^{1/l}$,
$\|\mathcal{L}_a^{-1}\| \approx |\eps|^{-1 + 1/l} $ into account,
we see that if we take $r = A |\eps|^{1/l}$ with $A$ sufficiently
small, it follows  from \eqref{lipschitztau} that $\mathcal{T}_a $ is a contraction of
a factor $1/10$ in the ball of radius $r$ for $|\eps|$
sufficiently small.

Now we try to identify the conditions that the ball $\mathcal{B}_r(0)$ with
$r$ chosen as above gets mapped into itself for small $\eps$.

If $r$ satisfies the conditions that make $\mathcal{T}_a(V)$ a contraction in
$\mathcal{B}_r(0)$, we have:
\begin{equation}\label{ballmapped}
\begin{split}
&\|\mathcal{T}_a(V)\|_{\rho, m}\\
&\leq\|\mathcal{T}_a(0)\|_{\rho, m}+ \|\mathcal{T}_a(V)-
\mathcal{T}_a(0)\|_{\rho, m} \\
&\leq \| \mathcal{L}_{a}^{-1}\|\left(
|\eps|\|\widetilde{f}(\theta, 0)\|_{\rho, m}+\|
h(\theta,a)\|_{\rho,m} +|\eps|\|g(\theta,a)\|_{\rho, m}\right) + r/10\\
& \le C|\eps|^{-1+1/l}\left(
|\eps|\|\widetilde{f}(\theta, 0)\|_{\rho, m} +  |\eps|^{1+2/l} +  |\eps|^{1+1/l} \right)
+ r/10.
\end{split}
\end{equation}

Therefore under the assumption that
\begin{equation} \label{newassumption}
\| \widetilde{f}(\theta, 0) \|_{\rho, m}
\end{equation}
is small enough we
obtain that
 $\mathcal{T}_a(\mathcal{B}_{r}(0)) \subset\mathcal{B}_{r}(0)$
and  we already had that   $\mathcal{T}_a$
is a contraction  in this ball.
\begin{remark} Note that the smallness assumption
\eqref{newassumption} depends on $| \overline{f} (0)|$. Indeed
a more detailed analysis shows that we could write
\eqref{newassumption} as $ \|\widetilde{f}(\theta,0)\|_{\rho, m}/|\overline{f}(0)| $
sufficiently small.
\end{remark}

It follows from the fixed point theorem in the Banach space  $H^{\rho,m}$ that
there exists a unique solution $V\in H^{\rho,m}$ for equation
 \eqref{fixeq1}. This produces a solution $x(t)=a+V(\theta)$ for equation \eqref{dege}.  Notice that, once we fix $a$, the $V$ is unique in the chosen ball.
This shows that the solution $x(t)=a+V(\theta)$ of \eqref{dege} is locally unique.

From the contraction mapping properties, we obtain easily regularity
with respect to parameters, since the regularity of solutions of
contraction mappings with parameters is standard. In particular, we
note that the contraction mapping for analytic families is very standard.

\subsection{The finitely differentiable case}

\begin{theorem}\label{theom-fi}
We study the equation \eqref{dege} with $h$ vanishing to order $(l+1)$ at zero.

	Suppose that  $f,\,h\in
	C^{m+p}(\mathbb{T}^d\times \mathbb{R},\mathbb{R}) \,\,(p=1,2,\cdots)$ and $\widetilde{f}(\theta,0)\in H^{m}(\mathbb{T}^d,\mathbb{R})$ with $m>\frac{d}{2}$.
	
	 If  $\overline{f}(0)\neq 0$ and
$\| \widetilde f(\theta,0)\|_{m}$ is sufficiently small compared
to $|\overline{f}(0)|$, then, there exists a $\eps_0 > 0 $ such that, defining
$\mathcal{I} = (-\eps_0, \eps_0)$ for $l$ odd
and $\mathcal{I} = (-\eps_0, 0)$ when $l$ even and
$\overline{f}(0) > 0$,
and $\mathcal{I} = (0, \eps_0)$ when $l$ even and
$\overline{f}(0) < 0$,
we have
that for all
$\eps \in \mathcal{I}$, there
exists a solution of \eqref{dege} of the form of \eqref{realso} in $H^{m}(\mathbb{T}^d,\mathbb{R})$.

	Moreover, the solution of equation \eqref{dege} is locally unique.
\end{theorem}
The same strategy presented for Theorem~\ref{mainthm} applies also to
the case that $f$ is finitely differentiable (but with sufficiently
high derivatives). Therefore, similar to the way in
Section~\ref{sec:analytic} and together with Lemma~\ref{gag-nir00} in
Sobolev case,  we can easily prove
Theorem~\ref{theom-fi}.

\begin{remark}
Using Remark~\ref{highder} we see that, in the analytic
case (resp. when $g$ is sufficiently differentiable),
 the operator $\mathcal{T}_a$ is analytic from $H^{\rho, m}$ to itself
(resp. several times differentiable from $H^m$ to itself) with
$m$ as in the main theorems.

Since the operator is differentiable with respect to $\eps$
we obtain that the solution produced depends analytically (resp.
differentially) on parameters.
\end{remark}

\section{The case of complex  $\eps$.
The phenomenon of monodromy}
\label{sec:complex}

The previous analysis has  shown that the leading term in
the solution \eqref{realso} is a constant. Note that we have shown
that  $\| V\|_{\rho,m}$ is much smaller than $|a|$.

The leading effect is the equation
\eqref{choosea}, which is an algebraic equation.
The study of the algebraic equation is much more natural when all
the variables are complex. Allowing complex values for $a$
makes superfluous to distinguish between odd and even $l$,
but it emphasizes that we can get more solutions.

Note that all the other arguments that
we have developed to compute the correction $V$ work just as well
when they are complex valued.

An elementary remark is that if we consider a closed  path in
the $\eps$ plane $\eps =  \alpha \exp{ 2 \pi \mathrm{i} t}$, $ t \in [0,1],\,\alpha\in \mathbb{R}$, we see that the solutions move only in a segment
\begin{equation*}
a = (-\overline{f}(0)\alpha)^{1/l}  \exp(2 \pi \mathrm{i} s),\,s \in [0,1/l].
\end{equation*}
Hence, if we continue $a$ while we vary $\eps$ along a circle,
the $a$ does not come to the same value.  If we repeat the path above
$l$ times ( $\eps =  \alpha \exp{ 2 \pi \mathrm{i}t}$, $ t \in [0,l]$), then $a$ gets back to the original value. This is the phenomenon
of monodromy.

When we consider the nonlinear problem, we observe that we can not apply the
contraction argument if $a$ is close to the imaginary axis.
On the other hand, in a region of the form $|\text{Im}(\eps)|
\le C |\text{Re}(\eps)|$, we obtain that
$\emph{dist}(a, \mathrm{i} \mathbb{R})$ is comparable with $|\eps|^{1/l}$.
These regions in complex $\eps$ are geometrically a ball
with $2 l$ cones removed.

In these regions, the argument developed in this paper applies and we
get the results.  The solutions depend differentially and they
are a small deformation of the solutions.  Hence the space of
the solutions contains a branch surface (minus some cuts).

Monodromy has appeared in other problems in degenerate perturbation theory
\cite{LlaveT94, JorbaLZ99}, but the regions excluded are a more elaborate
since the analysis is more elaborate.

We note that the fact that for complex $\eps$ we get several
solutions at a distance $\mathcal{O}(|\eps|^{1/l})$.
This shows that one cannot hope to obtain contraction in larger balls
by methods that work also for complex valued functions such as the soft
methods employed here.

\section{Higher dimensional phase space}\label{sec:high}
In this section, we consider the existence of response solutions for the $n$-dimensional quasi-periodically forced system:
\begin{equation}\label{dege2}
\begin{split}
\dot{x}=\phi(x) + h(\omega t, x)+\eps f(\omega t,x),
\end{split}
\end{equation}
where  $\phi: \mathbb{R}^n \rightarrow \mathbb{R}^n$
is a homogeneous function of degree $l$, i.e.
\begin{equation} \label{homogeneous}
\phi(\lambda x) =\lambda^l \phi(x), \quad  \lambda \in \mathbb{R}_+,\,\,
x \in \mathbb{R}^n,
\end{equation}
 and  $h$ vanishes to order
$(l+1)$ in $x$.
Of course, one important example of homogeneous functions is the
polynomials all of whose terms have degree $l$, but there are
other functions.  The polynomials are precisely those that are
$(l +1)$  times differentiable at the origin, but it is natural to consider
functions which are not differentiable at the origin.
We note that the form \eqref{dege2} appears naturally when we are
considering functions and expanding them in Taylor polynomials. We keep the lowest
degree.

Note that the range of $\phi$ will be always a cone.
We note also  that for a homogeneous function, taking derivatives of
\eqref{homogeneous}, we have Euler's formula:
\begin{equation} \label{euler}
(D \phi)(\lambda x)  = \lambda^{l -1} D \phi(x).
\end{equation}

The strategy is very similar to the one used when $n =1$.
We  assume \eqref{averagenonzero} (in the sense that $\overline{f}(0)=\overline{f}_j(0)$ with $\overline{f}_{j}(0)\neq 0\,(j=1,2,\cdots,n)$) and
that
\begin{equation} \label{rangeassumption}
\begin{split}
&\overline{f}(0) \in \text{Interior}( \text{Range}(\phi) ), \\
&\text{or} \\
&-\overline{f}(0) \in \text{Interior}( \text{Range}(\phi) ).
\end{split}
\end{equation}

In the first case of \eqref{rangeassumption}, we will obtain results for
all $0 < -\eps \ll 1$ and in the second case, we
will obtain results for $0 <\eps \ll 1$. Of course, both
cases can happen at the same time.  We will introduce the following notation, for a positive constant $\varepsilon_0$,
\begin{equation*}
\mathcal{I} = \begin{cases}  [0, \eps_0)\\
(-\eps_0, 0] \\
(-\eps_0, \eps_0)
\end{cases}
\end{equation*}
depending on whether  only  the first of
\eqref{rangeassumption} is true, only the second of \eqref{rangeassumption} is true
or both of  \eqref{rangeassumption} are true.

We indicate that the assumption \eqref{rangeassumption} is
an analogue in higher dimensions of the assumption \eqref{averagenonzero}
in the one dimensional phase space case.

Using \eqref{rangeassumption} in the second case,
 we will be able to find $a_0 \in \mathbb{R}^n$
such that $\phi(a_0) =- \overline{f}(0)$
and hence, $a = \eps^{1/l} a_0$ satisfies
$\phi(a) = -\eps \overline{f}(0)$ for positive $\eps$.
Analogously, in the first case of \eqref{rangeassumption},
we get $a$ defined for negative $\eps$. For simplicity of
notation, we will only discuss the second case from now on.
One can obtain the other case by changing $\eps$
to $-\eps$.

We note that, because of \eqref{euler},
\begin{equation}\label{euler1}
(D \phi)(a)  = \eps^{1 -1/l} D \phi(a_0).
\end{equation}
We will make the assumption that
\begin{equation}\label{spectrum}
\text{Spec}(D \phi(a_0) )  \cap  \mathrm{i} \mathbb{R} = \emptyset.
\end{equation}
Hence,
\[
\sup_{t \in \mathbb{R}} \| (\mathrm{i}t - D\phi(a_0) )^{-1} \| < \infty.
\]
And, using \eqref{euler1} we have
\[
\begin{split}
\sup_{t \in \mathbb{R}} \| (\mathrm{i}t - D\phi(a) )^{-1} \| =
\sup_{t \in \mathbb{R}} \| (\mathrm{i} t - \eps^{1 -1/l}D\phi(a_0) )^{-1} \|
 \le C \eps^{-1 + 1/l}.
\end{split}
\]

If  we define, as before
$\mathcal{L}_a$, we have
\[
\|\mathcal{L}_a^{-1} \| \le C
\eps^{-1 + 1/l}.
\]

Since the composition estimates are the same for higher
dimensional vectors as in the case of one dimensional vectors,
we follow exactly the proof of Theorem~\ref{mainthm}, Theorem~\ref{theom-fi}
and obtain:

\begin{theorem}\label{theom-fih}
Consider the equation \eqref{dege2} with  $h$ vanishing to order $(l+1)$ and  $f$ satisfying \eqref{averagenonzero} and \eqref{rangeassumption}.

 Assume that $\phi$ is homogeneous of
degree $l$, i.e. \eqref{homogeneous}.

If $f,h$ are analytic, $\widetilde f(\theta,0)\in H^{\rho, m}$ with $\rho>0,\,m > d$,
and $\| \widetilde f(\theta,0)\|_{\rho, m} $ is small enough,  then
for all $\eps \in \mathcal{I}$, we obtain a solution of
\eqref{dege2} of the form \eqref{realso} in $H^{\rho,m}$.

If $f,h$ are $C^{m +p} \,(p=1,2,\cdots)$, $\widetilde f(\theta,0) \in H^m$ with $m > d/2$, and $\| \widetilde f(\theta,0)\|_{m} $ is small enough, then
for all $\eps \in \mathcal{I}$, we obtain a solution of
\eqref{dege2} of the form \eqref{realso} in $H^{m}$.

Moreover, the solution of \eqref{dege2}
 is locally unique.
\end{theorem}

\begin{remark} We note that the method can be generalized
to the case that $\phi(x)$ is not a homogeneous function.
The key is that we can solve $\phi(a) =- \eps \overline{f}(0)$ and
that we can get bounds of $\| (\mathrm{i} t-D\phi(a) )^{-1}\|$.

This is  possible under several sets of conditions, such as $\phi$
being the sum of homogeneous functions, etc.
We will not explore these possibilities.
\end{remark}

\section{Results when the average forcing vanishes}
\label{sec:noaverage}
Both in our previous treatment and in \cite{ss18},
the assumption \eqref{averagenonzero} plays an important role.
In this section, we present some results without this assumption.
We will, however need other assumptions, such as Diophantine condition.

According to the heuristic principles we described
in Section~\ref{sec:guide}, the constant $a$ from solving $a^l + \eps\overline{f}(0)=0$ is the dominant part in \eqref{fixeq1}. 
This is based on the condition $\overline{f}(0)\neq0.$ In this part, we remove this condition. 
Therefore, we need to take the function $V$ from solving the homological equation $\partial_{\omega}V=\widetilde{f}(\theta,0)$ 
as the dominant part in \eqref{Uequation}. To deal with this equation we need some non-degeneracy assumptions, 
see \eqref{weakdiophantine}, which are much
weaker than Brjuno assumptions, for analytic case, and \eqref{dio}, which is the standard Diophantine assumptions, for finitely differentiable case.

As we will see,  our results have different
assumptions depending on whether  $l = 2$ or $ l > 2$.
The difference between
the two ranges of $l$ is real and not an artifact of the methods
since the solutions are somewhat different. 

Since the method is mainly algebraic manipulations and
contractions, it also leads
easily to results when the average is not zero but it
is  small compared
with other quantities that appear.
This goes in the opposite direction of
the results in the previous sections where  we assumed that other quantities
are small compared with the average.  We note that the solutions
we produce  in both cases are qualitatively different in the two regimes so
that it seems clear that there is  some bifurcation, but we
do not know how to formulate this precisely, much less to develop
a theory.

\subsection{Formulation of results in the zero average forcing  case}

\subsubsection{Description of the method for $l > 2$}
\label{descriptionl>2}
In this section we will describe the method we propose in
an informal way. We will ignore for the moment,
 precise definitions of spaces and formulating
precisely the hypotheses. This will be done immediately afterward, after
the steps to be taken are clarified. The informal assumption will clarify the
reasons for our choices.

We assume that in \eqref{dege}, we have $\overline{f}(0) = 0$.
We will try to find solutions of the form
\begin{equation} \label{hull}
x(t) = \eps V(\omega t) + U(\omega t).
\end{equation}

We choose
$V$ to solve the (dominant) equation
\begin{equation} \label{dominant}
\omega \partial_\theta V = \widetilde f(\theta,0).
\end{equation}
The reason that we impose the condition \eqref{weakdiophantine} in the analytic case is that we
will solve the equation \eqref{dominant}, whose small divisor is $\mathrm{i}(k\cdot\omega).$
With the estimate \eqref{weakdiophantine}, by shrinking
the complex domain $\rho$ to $(\rho-\eta)$ we can guarantee the solution to this equation is controllable.
As for finitely differentiable case, since there is no complex domain, we have to lose the regularity $m.$
In this case the condition \eqref{weakdiophantine} is not enough, we need the standard Diophantine condition
\eqref{dio}.

Notice that the Diophantine conditions
\eqref{weakdiophantine} are much weaker
than the assumptions in KAM theory.
The reason is that in our case, we only need
to solve small divisor equations twice. Hence,
we can afford that they have a more drastic
effect than in KAM theory where one needs to
solve infinitely many small divisor equations
as part of an iterative process. In our case,
we solve small divisor equations only to
set up a contraction argument.

Also, we note that the solutions of \eqref{dominant} will never be unique
since we can add a constant.  In what follows, we will assume that we
have chosen the $V$ and transform the equation for the fluctuation
accordingly. We will not revisit the choice of $V$ (except at the end
of the discussion in Section~\ref{descriptionl=2}, where we will find
that there is an advantage in choosing the constant so that
$\overline{g_1} + 2 \overline{V} \ne 0$).

In this section, we will assume that  $l > 2$. As we will see in
Section~\ref{descriptionl=2},
the case $l =2$ leads to a different answer with different non-degeneracy conditions.

We will find it convenient to introduce some notation for
the expansions of $g$ in the second variable (of course, this is
just continuing the expansion of the forcing $f$, but we will keep the notation
$g$ we used before)
\[
g(\theta,x)  = g_1(\theta)x + g_>(\theta, x),
\]
where, of course, $g_>(\theta, 0) = 0$, $D_x g_> (\theta, 0) = 0$.

Once we have chosen the function $V$ solving
\eqref{dominant}, $x(t)$ given by \eqref{hull} solves
\eqref{dege} if and only if $U$ solves
\begin{equation} \label{Uequation}
\begin{split}
\omega \partial_\theta U = (U + \eps V)^l
+ h(\theta, U + \eps V)+\eps g_1(\theta)(U+\eps V) +
\eps g_>(\theta, U + \eps V).
\end{split}
\end{equation}

We will see that the main  part of the equation
\eqref{Uequation} is the following:
\[
\mathcal{M}U  \equiv
 \omega \partial_\theta U -   \eps g_1(\theta)U.
\]
As indicated in the sketch of the strategy, we will
try to invert $\mathcal{M}$ to formulate \eqref{Uequation}
as a fixed point equation.

As we will see more precisely in Lemma~\ref{inverse},
the operator $\mathcal{M}$ can be inverted provided
that
\begin{equation} \label{secondarynondeg}
\overline{g_1} \ne 0
\end{equation}
as well as some very weak Diophantine equations
and we can obtain bounds in  the Sobolev spaces
we have used in the previous sections.
Then, the equation~\eqref{Uequation} is
equivalent to
\begin{equation}\label{fixedU}
U  = \mathcal{M}^{-1}
\left( (U + \eps V)^{l} + h(\theta, U +\eps V) +\eps^2 g_1(\theta)V+ \eps  g_>(\theta, U+ \eps V)\right),
\end{equation}
which is of a form very similar to \eqref{fixeq3}.

Once we have the estimates for $\mathcal{M}$,
the Lipschitz properties of
the non-linear terms can be estimated
rather easily when $l > 2$. As it turns out,
the term $(U + \eps V)^{l}$ has very small Lipschitz constant
when $l$ is larger. When $l =2$, we will have to rearrange the
equations a bit more. See Section~\ref{descriptionl=2}.

\begin{remark} It is a natural question
to ask what will happen if the average forcing is
zero and \eqref{secondarynondeg} fails.
It seems plausible that one can make progress identifying
other leading terms which will have to vanish and solve
the auxiliary equation. Eliminating the assumption
 \eqref{averagenonzero}
seems to bring in the qualitatively different assumption
\eqref{weakdiophantine}, but higher order non-degeneracy seems
to bring no new phenomenon.
\end{remark}

\subsubsection{Description of the method for $l = 2$}
\label{descriptionl=2}

As before, we start by a heuristic description of the method.
We will keep as much of the notation introduced in
Section~\ref{descriptionl>2}.

As we will see, the conditions we need are different since
the dominant terms that we need to consider are different.

In the case $l = 2$, we will rewrite \eqref{Uequation} (which is equivalent
for \eqref{dege} with the notations introduced)
\begin{equation} \label{Uequation2}
\begin{split}
\omega \partial_\theta U &=U^2  + 2 \eps V U + \eps^2 V^2+ h(\theta, U + \eps V)¡¢\\¡¢
&\ \ \ \ \ \ \ + \eps g_1(\theta)(U+ \eps V(\theta)) +\eps
g_>(\theta, U + \eps V(\theta))
\end{split}
\end{equation}
which is equivalent to:
\begin{equation} \label{Uequation3}
\begin{split}
\big(\omega \partial_\theta & -  (\eps g_1(\theta) +2 \eps V) \big) U \\
& =
  U^2  + \eps^2 V^2
+ h(\theta, U + \eps V)+\eps^2 g_1(\theta)V+\eps g_>(\theta, U + \eps V(\theta).
\end{split}
\end{equation}

We proceed to invert  the operator $\mathcal{N}$ defined by
\[
\mathcal{N}U  \equiv
\left(\omega \partial_\theta  -  (\eps g_1(\theta) + 2 \eps V) \right) U
\]
which can be done in the same way as we inverted $\mathcal{M}$ since they
are operators of the same form. The difference of $\mathcal{M}$ and
$\mathcal{N}$ is that $\mathcal{N}$ contains an extra multiplication
term.  Then, the equation \eqref{Uequation3} can be transformed into
\begin{equation} \label{findU2}
\begin{split}
U = \mathcal{N}^{-1} \left(U^2  + \eps^2 V^2
+ h(\theta, U + \eps V)+\eps^2 g_1(\theta)V+\eps g_>(\theta, U + \eps V(\theta)\right).
\end{split}
\end{equation}

Following the procedure in Lemma~\ref{inverse},
the operator $\mathcal{N}$ can be inverted provided
that we have that
\begin{equation} \label{secondarynondegl=2}
\overline{g_1} + 2 \overline{V}  \ne 0.
\end{equation}

The equation \eqref{secondarynondegl=2} appears for the same
reasons as \eqref{secondarynondeg}. We note however that
\eqref{secondarynondegl=2} can always be arranged if we choose, from
the beginning the $V$ solving \eqref{dominant} taking advantage of
the lack of uniqueness of solutions of \eqref{dominant}. Adding
an arbitrary constant to them is always possible, so that
\eqref{secondarynondegl=2} can always be satisfied.

Of course, the choice of $\overline{V}$ will affect some of the details
of subsequent calculations and it will affect the value of $\eps_0$
which determines the maximum size of the perturbations allowed but
will not affect the qualitative arguments.

\begin{remark}
The reason why the case $l = 2$ is special is because the
linear approximation of $U$ in $(U + \eps V)^l$  for general $l$ is
$l \eps^{l -1} V^{l -1}$. We see that in the case that $l =2$ this is
a term of order $\eps$ of the same order of magnitude as
$\eps g_1$, When $l > 2$, the linear in $U$ approximation of
$(U + \eps V)^l$ is much smaller than the $\eps g_1$.
\end{remark}

\subsection{Precise formulation of the main results in the zero average case}

\begin{theorem}\label{main-zero}
Consider the differential equation of the form \eqref{dege} with  $h$ vanishing to order $(l+1)$
 and  $\overline{f}(0)=0$.

Assume that:
\begin{itemize}
	\item
	$f,h$ are analytic, $\widetilde f(\theta,0)\in H^{\rho, m}$ with $\rho>0,\,m > d$,
	and $\| \widetilde f(\theta,0)\|_{\rho, m} $ is small enough.
	\item
	The frequency $\omega$ satisfies \eqref{weakdiophantine} with
	some $\eta>0$ smaller than $\rho$.
	\item
	In the case that $l > 2$,  the average of $g_1$ is not zero.
\end{itemize}
 Then,
for all $\eps \in \mathcal{I}$, we obtain a solution of
\eqref{dege} of the form \eqref{hull} in $H^{\rho-\eta,m}$.

We also have that 
 if $f,h$ are $C^{m +p} \,(p=1,2,\cdots)$, $\widetilde f(\theta,0) \in H^m$ with $m > d/2$ and $\omega$ satisfies \eqref{dio} with some $\tau$ satisfying $d-1<\tau<m$, then we obtain a solution in $H^{m-\tau}$.
\end{theorem}

Since the proof is based on contraction mappings, we also obtain
local uniqueness and smooth dependence on parameters. We leave 
the straightforward formulation to the reader.

\subsection{Some auxiliary lemmas}\label{sec:dio}
In this section, we present some auxiliary lemmas motivated
by the  sketch of the arguments indicated in
Section~\ref{descriptionl>2} and \ref{descriptionl=2}.
They will allow  to carry out all the estimates required in
the sketch and make  it rigorous.

\begin{lemma}\label{cohomologyexpo}
For some $\rho,\eta > 0,$ if
 the frequency vector $\omega$
satisfies
\begin{equation} \label{weakdiophantine}
|k \cdot \omega| \ge \gamma \exp( - \eta |k| ), \quad \emph{for}\,\, k \in \mathbb{Z}^d \setminus \{0\},
\end{equation}
with $0<\gamma\ll1,$ then, for $\rho > \eta$, we have that if  $f \in H^{\rho, m}$ has zero
average, then  there is a unique solution $V$ of zero average of
the equation
\begin{equation}\label{cohomologyloss}
\omega \cdot \partial_\theta V = f.
\end{equation}

Moreover, we have  $V \in H^{\rho - \eta, m}$ and
\begin{equation}\label{expbounds}
 \| V \|_{\rho -\eta,m} \le   \gamma^{-1} \| f \|_{\rho, m}.
\end{equation}
\end{lemma}

\begin{proof}
The proof is obvious if we realize that the equation \eqref{cohomologyloss} is
equivalent to  the Fourier coefficients of $V$ as the following:
\[
\mathrm{i}  (k \cdot \omega) \, \hat V_k =  \hat f_k, \quad \text{for}\,\,k \in \mathbb{Z}^d.
\]
This determines  $\hat V_k$ when $k \ne 0$ and normalizing $V$ to
zero average gives $\hat V_0 = 0$. Then,
\eqref{expbounds} establishes since
the norm of $V$ in the space $H^{\rho -\eta , m}$
is read off the size  of the Fourier coefficients of $V$.
\end{proof}

\begin{lemma}\label{inverse}
For $\rho,\,\gamma,\,\eta  > 0$ with $\rho>\eta$ and $m > d$, let $\omega$ satisfy \eqref{weakdiophantine} and $\beta \ne 0$ be a real constant.

 If  $\varphi \in H^{\rho, m}$ have zero average, then, for any $f \in H^{\rho-\eta, m}$, there is a unique solution $V \in H^{\rho-\eta,m}$
solving
\begin{equation}\label{twistedcohomology}
(\omega \partial_\theta  + \beta+ \varphi)V  = f.
\end{equation}

Furthermore, we have
\begin{equation}\label{twistedcohomologyestimates}
\| V\|_{\rho-\eta , m }
\le |\beta|^{-1}
\|f\|_{\rho-\eta, m} \exp( 2 \gamma^{-1} \| \varphi\|_{\rho,m}).
\end{equation}
\end{lemma}

\begin{remark}
Note that the Lemma~\ref{inverse} would be immediate under the extra assumption
that $|\gamma|^{-1} \| \varphi \|_{\rho+\eta, m} $ sufficiently small. In such
a case we could invert   the operator
$(\omega \partial_\theta  + \beta)$  using Fourier series and
then use the Neumann series to invert
$(\omega \partial_\theta  + \beta + \varphi)$. For our applications,
it is desirable not to make the extra  assumption.
\end{remark}

\begin{remark}
Equations of the form \eqref{twistedcohomology} are called
\emph{``twisted cohomology equations"} in \cite{Herman83},
which also develops techniques to solve them.

There are several interesting variants
of  \eqref{twistedcohomology} estimates.
\end{remark}

\begin{proof}
The proof is very similar to the integrating factor method in
linear ODE's.

We find $\Gamma$ solving $\omega \partial_\theta \Gamma = \varphi$
(as in the case of the integrating factor, we remark that such
$\Gamma$ is unique up to an additive constant).

By Lemma~\ref{cohomologyexpo}, we have
\[
\| \Gamma \|_{\rho-\eta, m}  \le  \gamma^{-1}  \|\varphi \|_{\rho ,m}
\]
and, by  the Banach algebra properties of the Sobolev norm,
\[
\|\exp( \Gamma ) \|_{\rho-\eta, m}  \le
\exp( \gamma^{-1}  \|\varphi \|_{\rho,m}).
\]

Then, multiplying \eqref{twistedcohomology}  by $\exp(\Gamma)$,
we obtain that it is equivalent to
\[
\begin{split}
 f \exp(\Gamma) &=
\exp(\Gamma) \omega \partial_\theta V  + \beta V \exp(\Gamma)
+ \exp(\Gamma) (\omega\partial_\theta \Gamma) V  \\
&= \left( \omega \cdot \partial_\theta + \beta \right)( \exp(\Gamma) V).
\end{split}
\]

Hence, using that the operator
$ \omega \cdot \partial_\theta + \beta$ is invertible
and so are the operators of multiplication by
$\exp( \pm \Gamma)$, one has
\begin{equation} \label{twistedcohomologysolution}
V = \exp( - \Gamma) \left( \omega \cdot \partial_\theta + \beta\right)^{-1}
f\exp(\Gamma).
\end{equation}

{From} \eqref{twistedcohomologysolution}, the estimates claimed in
\eqref{twistedcohomologyestimates}  follow immediately.
\end{proof}
\begin{remark}
In case that $\beta = 0$, to use formula \eqref{twistedcohomologysolution}
we need to assume that $\exp(\Gamma) f$ has average zero.
This shows that in this case we will require different arguments.
\end{remark}

\begin{remark}
	Note that Lemma~\ref{cohomologyexpo} and Lemma~\ref{inverse} are aimed at the analytic functions. When we consider our problem in finitely differentiable setting, we need to assume that the frequency $\omega$ satisfies
	\begin{equation} \label{dio}
	|k \cdot \omega| \ge \gamma |k|^{-\tau}, \quad \emph{for}\,\, k \in \mathbb{Z}^d \setminus \{0\}
	\end{equation}
with $d-1<\tau<m$ and $0<\gamma\ll1$ (the condition $\tau>d-1$ guarantees that the set whose elements are the frequencies satisfying 
\eqref{dio} is of positive Lebesgue measure).
	Then, for $f \in H^{m},\,m>\frac{d}{2}$ has zero
	average,  there is a unique solution $V\in H^{m-\tau}$ of zero average of
	the equation
	\begin{equation*}
	\omega \cdot \partial_\theta V = f
	\end{equation*}
	satisfying
	\begin{equation}\label{expbounds1}
	\| V \|_{m-\tau} \le   \gamma^{-1} \| f \|_{m}.
	\end{equation}
	Moreover,
	for $\psi \in H^{m-\tau}$ have zero average, then, for any $f \in H^{m-\tau}$, there is a unique solution $V \in H^{m-\tau}$
	solving
	\begin{equation*}
	(\omega \partial_\theta  + \beta+ \varphi)V  = f
	\end{equation*}
	with
	\begin{equation*}
	\| V\|_{m-\tau }
	\le |\beta|^{-1}
	\|f\|_{m-\tau} \exp( 2 \gamma^{-1} \| \varphi\|_{m}).
	\end{equation*}
\end{remark}

\subsection{Proof of the results in the zero average forcing case}
We only present the detailed proof of Theorem~\ref{main-zero} in analytic case. The finitely differentiable case is similar.
\subsubsection{The case $l > 2$}
In the case $l > 2$, we will consider the equation \eqref{fixedU}
and check the hypotheses of the contraction mapping principle for
the operator on the right.

The operator $\mathcal{M}$ fits into Lemma~\ref{inverse}
by taking $\beta = -\eps\overline{g_1}$,
$\varphi = -\eps\widetilde{g_1}$. Therefore we obtain
$\| \mathcal{M} \| \le C_1 |\eps|^{-1}$, where $C_1$ depends on $
g_1,\gamma$. To simplify the notation, we still use  $C_1$ to represent all constants (may depend on $l,\gamma,f,h$ but not depend on $\eps$).

Recall that \eqref{fixedU} is an equation for $U$ and that
$V$ has already been picked.

If we consider a ball of radius $r$ with $r \le A |\eps|$,
(we henceforth fix $A$, so that all the constants may depend on it),
we can estimate the Lipschitz constants of the nonlinear terms in
the right hand side of \eqref{fixedU} with respect to the
$U$ variable as the following:
\[
\begin{split}
& \text{Lip}_U\left( (U + \eps V)^l   \right) \le C_1 |\eps|^{l-1}, \\
& \text{Lip}_U\left( \eps h(\theta, U +\eps V)   \right) \le C_1 |\eps|^l,\\
&\text{Lip}_U\left( \eps g_>( U +\eps V)   \right) \le C_1 |\eps|^2.
\end{split}
\]
 Note that the distance is measured in $H^{\rho -\eta, m}$.

Hence, we obtain that, the right hand side of \eqref{fixedU} has
a Lipschitz constant bounded by $C_1 |\eps|$. We choose $|\eps|$
small enough so that we get a contraction by $1/10$.

We also observe that for $U = 0$, the norm of the right hand side of
\eqref{fixedU} is bounded from above by
 $C_1|\eps|^{-1}( |\eps|^l + |\eps|^{l+1}+|\eps|^2|\widetilde f(\theta,0)|+|\eps|^3)$. Since we assume that $|\widetilde f(\theta,0)|$ is small enough, we get the ball to map into itself.

\subsubsection{The case l =  2}
The case $l = 2$ is based on the analysis of the operator in
the right hand side of \eqref{findU2}.

This is actually easier than the case of $l > 2$. By Lemma~\ref{inverse}, we have that
$\|\mathcal{N} \| \le C_1 |\eps|^{-1}$.

The Lipschitz constant of most  nonlinear terms in a ball or
radius $r = A |\eps|$ are estimated the same.
The only difference is
that we have
\begin{equation*}
\begin{split}
 \text{Lip}_U (U^2) \le C_1A |\eps|.
 \end{split}
\end{equation*}

Hence, we have that the Lipschitz constant of the right hand side of \eqref{fixedU} in the ball of
radius $A|\eps|$ can be made smaller than $1/10$ by taking $A$ small enough.

We also have that the $\|\cdot\|_{\rho - \eta, m}$
 norm of the right hand side of \eqref{fixedU} at $U = 0$  can be estimated
by $C_1 (|\eps|^{-1}(|\varepsilon|^2 |\widetilde{f}(\theta,0)|+|\eps|^3)$. Thus, by taking $|\eps|$ small enough, we
can get that  the operator maps the ball into itself.

\section{Application to degenerate oscillators (second order equations)}
\label{sec:oscillators}
Remarkably similar methods can be applied to
the study of degenerate oscillators (second order equations).
\begin{equation}
\label{degeosc}
\ddot{x} + \delta \dot{x}  = x^l + h(\omega t, x) + \eps f(\omega t, x)
\end{equation}
where $h,f$ are as in \eqref{dege}.
Again, we aim to find solutions of the form
\eqref{realso}.

Note that the equation  \eqref{degeosc} has two small parameters
$\delta$, $\eps$. Depending on the relation among them, we
will have  that the dominant solution has different forms.

In this paper, we only aim to demonstrate the possibilities of
the method  and will only do one of the cases. We hope to
come back to a more complete study.

A sample result is the following:
\begin{theorem}
	Consider the equation \eqref{degeosc}
	with $h,f$ as in \eqref{dege}.
	
	Assume that there exist  $a$ solving
	\[
	a^l + \eps \overline{f}(0) = 0
	\]
	and choose one of them.

	Assume:
	\begin{itemize}
		\item
		\eqref{averagenonzero}.
		\item
		$\|\widetilde f(\theta,0)\|_{\rho,m}$ is small enough compared
		to $|\overline{f}(0)|.$
		\item
		\[
		\delta^2 + 2 l a^{l-1}  \geq 0.
		\]
	\end{itemize}
	Then, the same conclusions
	as in Theorem~\ref{mainthm},~\ref{theom-fi}.
\end{theorem}

The proof is extremely similar to the study of
\eqref{dege}.
After we substitute \eqref{realso} in \eqref{degeosc} and
cancel $a^l + \eps\overline{f}(0)$
we see that \eqref{realso} is a solution of \eqref{degeosc} if
and only if $V$ satisfies:
\begin{equation}\label{tosolveosc}
\widetilde{\mathcal{L}}V =
S(a,V) +
\eps \widetilde{f}(0,0)
+h(\theta, a + V(\theta))+ \eps g(\theta, a + V(\theta) ),
\end{equation}
where
\[
\widetilde{\mathcal{L}} V  \equiv
\left[
(\omega\cdot \partial_\theta)^2
+ \delta
(\omega\cdot \partial_\theta) - l a^{l-1} \right] V.
\]
If the operator $\widetilde{\mathcal{L}}$ was invertible,
\eqref{tosolveosc} would be equivalent to
\begin{equation}\label{fixed1osc}
V = \widetilde{\mathcal{L}}^{-1} \left(
S(a,V) +
\eps \widetilde{f}(0,0)
 + h(\theta, a + V(\theta)) + \eps g(\theta, a + V(\theta) )\right).
\end{equation}
Note the similitude between \eqref{fixed1osc} and
\eqref{fixeq3}. The only difference is
the linear operator to be inverted.

Hence, we will need to study the invertibility of the operator
$\widetilde{\mathcal{L}}$ and the norm of its inverse.
We note that the operator $\widetilde{\mathcal{L}}$ is diagonal
in Fourier series and it amounts to multiplying the
$k$ Fourier coefficient by
\[
\widetilde{L}_k \equiv -(k \cdot \omega)^2 +  \mathrm{i}\delta (k \cdot \omega) - l a^{l-1}.
\]
Hence, to estimate $\| \widetilde{\mathcal{L}}^{-1}\|$,
it suffices to estimate from below  the minimum of $|\widetilde{L}_k|$.
Denoting $t = k \cdot \omega$, we have
\[
\begin{split}
|\widetilde{L}_k|^2 &=  (-t^2 - l a^{l-1} )^2 + \delta^2 t^2 \cr
&= t^4 + t^2(\delta^2 + 2 l a^{l-1}) + l^2 a^{2(l-1)}  \cr
&\ge | l a^{l-1}|^2,
\end{split}
\]
where the last inequality comes from the assumption
that $(\delta^2 + 2 l a^{l-1}) \ge 0$.

Once we have that, we see that the operator in
\eqref{fixed1osc} satisfies exactly the same bounds
as the operator in \eqref{fixeq3} and the rest of the proof
does not need any modification from the estimates
in the proof of Theorem~\ref{mainthm} (see \eqref{lipschitztau},
\eqref{ballmapped}).

\bibliographystyle{alpha}
\bibliography{dege-fixed-point}

\end{document}